\documentclass[intlimits]{amsart}

\usepackage[latin1]{inputenc}
\usepackage{amsmath,amsfonts,amssymb}
\usepackage{enumerate}

\newtheorem{problem}{Problem}

\newtheorem{theorem}{Theorem}
\newtheorem{lemma}[theorem]{Lemma}
\newtheorem{proposition}[theorem]{Proposition}
\newtheorem*{conjecture}{Conjecture}
\newtheorem*{corollary}{Corollary}
\newtheorem{question}{Question}

\newtheorem*{definition}{Definition}
\newtheorem*{remark}{Remark}
\newtheorem*{conclusion}{Conclusion}

\newcommand{\R}{\mathbb{R}}
\newcommand{\Z}{\mathbb{Z}}
\newcommand{\Q}{\mathbb{Q}}
\newcommand{\C}{\mathbb{C}}

\DeclareMathOperator{\Reg}{Reg}

\author{Fabrizio Barroero}
\email{barroero@math.tugraz.at}

\author{Christopher Frei}
\email{frei@math.tugraz.at}

\author{Robert F. Tichy}
\email{tichy@tugraz.at}

\address{Institut f\"ur Mathematik A\\
Technische Universit\"at Graz\\
Steyrergasse 30, A-8010 Graz\\
Austria}

\title[Additive unit representations]{Additive unit representations in global fields -- A survey}

\dedicatory{Dedicated to K\'alm\'an Gy\H{o}ry, Attila Peth\H{o}, J\'anos Pintz and Andr\'as Sark\"ozy.}



\thanks{F. Barroero is supported by the Austrian Science Foundation (FWF) project W1230-N13.\\
C. Frei is supported by the Austrian Science Foundation (FWF) project S9611-N23.}

\begin{document}

\begin{abstract}
We give an overview on recent results concerning additive unit representations. Furthermore the solutions of some open questions are included. The central problem is whether and how certain rings are (additively) generated by their units. This has been investigated for several types of rings related to global fields, most importantly rings of algebraic integers. We also state some open problems and conjectures which we consider to be important in this field.
\end{abstract}

\maketitle

\section{The unit sum number}

In 1954, Zelinsky \cite{Zelinsky1954} proved that every endomorphism of a vector space $V$ over a division ring $D$ is a sum of two automorphisms, except if $D = \Z/2\Z$ and $\dim V = 1$. This was motivated by investigations of Dieudonn{\'e} on Galois theory of simple and semisimple rings \cite{Dieudonne1948} and was probably the first result about the additive unit structure of a ring.

Using the terminology of V\'amos \cite{Vamos2005}, we say that an element $r$ of a ring $R$ (with unity $1$) is \emph{$k$-good} if $r$ is a sum of exactly $k$ units of $R$. If every element of $R$ has this property then we call $R$ \emph{$k$-good}. By Zelinsky's result, the endomorphism ring of a vector space with more than two elements is $2$-good. Clearly, if $R$ is $k$-good then it is also $l$-good for every integer $l > k$. Indeed, we can write any element of $R$ as
$$r = (r - (l-k)\cdot 1) + (l-k)\cdot 1,$$
and expressing $r - (l-k)\cdot 1$ as a sum of $k$ units gives a representation of $r$ as a sum of $l$ units.

Goldsmith, Pabst and Scott \cite{Goldsmith1998} defined the \emph{unit sum number} $u(R)$ of a ring $R$ to be the minimal integer $k$ such that $R$ is $k$-good, if such an integer exists. If $R$ is not $k$-good for any $k$ then we put $u(R) := \omega$ if every element of $R$ is a sum of units, and $u(R) := \infty$ if not. We use the convention $k < \omega < \infty$ for all integers $k$.

Clearly, $u(R) \leq \omega$ if and only if $R$ is generated by its units. Here are some easy examples from \cite{Goldsmith1998}:
\begin{itemize}
 \item $u(\Z) = \omega$,
 \item $u(K[X]) = \infty$, for any field $K$,
 \item $u(K) = 2$, for any field $K$ with more than $2$ elements, and
 \item $u(\Z/2\Z) = \omega$. 
\end{itemize}

Goldsmith, Pabst and Scott \cite{Goldsmith1998} were mainly interested in endomorphism rings of modules. For example, they proved independently from Zelinsky that the endomorphism ring of a vector space with more than two elements has unit sum number $2$, though they mentioned that this result can hardly be new.

Henriksen \cite{Henriksen1974} proved that the ring $M_n(R)$ of $n\times n$-matrices ($n\geq 2$) over any ring $R$ is $3$-good.

Herwig and Ziegler \cite{Herwig2002} proved that for every integer $n \geq 2$ there exists a factorial domain $R$ such that every element of $R$ is a sum of at most $n$ units, but there is an element of $R$ that is no sum of $n-1$ units.

The introductory section of \cite{Vamos2005} contains a historical overview of the subject with some references. We also mention the survey article \cite{Srivastava2010} by Srivastava.

In the following sections, we are going to focus on rings of ($S-$)integers in global fields.

\section{Rings of integers}

The central result regarding rings of integers in number fields, or more generally, rings of $S$-integers in global fields ($S \neq \emptyset$ finite), is that they are not $k$-good for any $k$, thus their unit sum number is $\omega$ or $\infty$. This was first proved by Ashrafi and V\'amos \cite{Ashrafi2005} for rings of integers of quadratic and complex cubic number fields, and of cyclotomic number fields generated by a primitive $2^n$-th root of unity. They conjectured, however, that it holds true for the rings of integers of all algebraic number fields (finite extensions of $\Q$). The proof of an even stronger version of this was given by Jarden and Narkiewicz \cite{Jarden2007} for a much more general class of rings:

\begin{theorem}\cite[Theorem 1]{Jarden2007}\label{JN}
If $R$ is a finitely generated integral domain of zero characteristic then there is no integer $n$ such that every element of $R$ is a sum of at most $n$ units.

In particular, we have $u(R) \geq \omega$, for any ring $R$ of integers of an algebraic number field.
\end{theorem}

This theorem is an immediate consequence of the following lemma, which Jarden and Narkiewicz proved by means of Evertse and Gy\H{o}ry's \cite{Evertse1988} bound on the number of solutions of $S$-unit equations combined with van der Waerden's theorem \cite{VanderWaerden1927} on arithmetic progressions.

\begin{lemma}\cite[Lemma 4]{Jarden2007}\label{JN_lemma}
If $R$ is a finitely generated integral domain of zero characteristic and $n \geq 1$ is an integer then there exists a constant $A_n(R)$ such that every arithmetic progression in $R$ having more than $A_n(R)$ elements contains an element which is not a sum of $n$ units.
\end{lemma}

Lemma \ref{JN_lemma} is a special case of a theorem independently found by Hajdu \cite{Hajdu2007}. Hajdu's result provides a bound for the length of arithmetic progressions in linear combinations of elements from a finitely generated multiplicative subgroup of a field of zero characteristic. Here the linear combinations are of fixed length and only a given finite set of coefficient-tuples is allowed. Hajdu used his result to negatively answer the following question by Pohst: Is it true that every prime can be written in the form $2^u \pm 3^v$, with non-negative integers $u$, $v$?

Using results by Mason \cite{Mason1986, Mason1986B} on $S$-unit equations in function fields, Frei \cite{Frei2010} proved the function field analogue of Theorem \ref{JN}. It holds in zero characteristic as well as in positive characteristic.

\begin{theorem}
Let $R$ be the ring of $S$-integers of an algebraic function field in one variable over a perfect field, where $S \neq \emptyset$ is a finite set of places. Then, for each positive integer $n$, there exists an element of $R$ that can not be written as a sum of at most $n$ units of $R$. In particular, we have $u(R) \geq \omega$.
\end{theorem}

We will later discuss criteria which show that in the number field case as well as in the function field case, both possibilities $u(R) = \omega$ and $u(R) = \infty$ occur.

\section{The qualitative problem}

\begin{problem}\cite[Problem A]{Jarden2007}\label{qual_problem}
Give a criterion for an algebraic extension $K$ of the rationals to have the property that its ring of integers $R$ has unit sum number $u(R)\leq \omega$.
\end{problem}

Jarden and Narkiewicz provided some easy examples of infinite extensions of $\Q$ with $u(R) \leq \omega$: By the Kronecker-Weber theorem, the maximal Abelian extension of $\Q$ has this property. Further examples are the fields of all algebraic numbers and all real algebraic numbers.

More criteria are known for algebraic number fields of small degree. Here, the only possibilities for $u(R)$ are $\omega$ and $\infty$, by Theorem \ref{JN}. For quadratic number fields, Belcher \cite{Belcher1974}, and later Ashrafi and V\'amos \cite{Ashrafi2005}, proved the following result:

\begin{theorem}\cite[Lemma 1]{Belcher1974}\cite[Theorems 7, 8]{Ashrafi2005}
Let $\Q(\sqrt{d})$, $d \in \Z$ squarefree, be a quadratic number field with ring of integers $R$. Then $u(R) = \omega$ if and only if 
\begin{enumerate}[1.]
 \item $d \in \{-1, -3\}$, or
 \item $d > 0$, $d \not\equiv 1 \mod 4$, and $d+1$ or $d-1$ is a perfect square, or
 \item $d > 0$, $d \equiv 1 \mod 4$, and $d+4$ or $d-4$ is a perfect square.
\end{enumerate}
\end{theorem}

A similar result for purely cubic fields was found by Tichy and Ziegler \cite{Tichy2007}.

\begin{theorem}\cite[Theorem 2]{Tichy2007}
Let $d$ be a cubefree integer and $R$ the ring of integers of the purely cubic field $\Q(\sqrt[3]{d})$. Then $u(R) = \omega$ if and only if
\begin{enumerate}[1.]
 \item $d$ is squarefree, $d \not\equiv \pm 1 \mod 9$, and $d + 1$ or $d - 1$ is a perfect cube, or
 \item $d = 28$.
\end{enumerate}
\end{theorem}

Filipin, Tichy and Ziegler used similar methods to handle purely quartic complex fields $\Q(\sqrt[4]{d})$. Their criterion \cite[Theorem 1.1]{Filipin2008} states that $u(R) = \omega$ if and only if $d$ is contained in one of six explicitly given sets.

As a first guess, one could hope to get information about the unit sum number of the ring of integers of a number field $K$ by comparing the regulator and the discriminant of $K$. In personal communication with the authors, Martin Widmer pointed out the following sufficient criterion for the simple case of complex cubic fields:

\begin{proposition}(Widmer)
If $R$ is the ring of integers of a complex cubic number field $K$ then $u(R) = \omega$ whenever the inequality
\begin{equation}\label{disc_reg}|\Delta_K| > (e^{\frac{3}{4}R_K} + e^{-\frac{3}{4}R_K})^4\end{equation}
holds. Here, $\Delta_K$ is the discriminant and $R_K$ is the regulator of $K$.
\end{proposition}

\begin{proof}
Regard $K$ as a subfield of the reals, and let $\eta > 1$ be a fundamental unit, so $R_K = \log \eta$. Since $K$ contains no roots of unity except $\pm 1$, the ring of integers $R$ is generated by its units if and only if $R = \Z[\eta]$. By the standard embedding $K \to \R \times \C \simeq \R^3$, we can regard $R$ and $\Z[\eta]$ as lattices in $\R^3$ and compare their determinants. Let $\eta'=x+iy$ be one of the non-real conjugates of $\eta$. We get $u(R) = \omega$ if and only if
$$2^{-1} \sqrt{|\Delta_K|} = \left| \det\begin{pmatrix}1 & \eta & \eta^2\\ 1 & x & x^2 - y^2\\ 0 & y  & 2 x y \end{pmatrix}\right|\text.$$
Since the right-hand side of the above equality is always a multiple of the left-hand side, we have $u(R) = \omega$ if and only if 
$$\sqrt{|\Delta_K|} > \left| \det\begin{pmatrix}1 & \eta & \eta^2\\ 1 & x & x^2 - y^2\\ 0 & y  & 2 x y \end{pmatrix}\right|\text.$$
Clearly, $\eta^{-1} = \eta' \overline{\eta'} = x^2 +y^2$, whence $|x|$, $|y| \leq \eta^{-1/2}$. With this in mind, a simple computation shows that the right-hand side of the above inequality is at most $\eta^{-3/2} + 2 + \eta^{3/2}$, so \eqref{disc_reg} implies that $u(R) = \omega$.
\end{proof}

To see that condition \eqref{disc_reg} is satisfied in infinitely many cases, we consider the complex cubic fields $K_N = \Q(\alpha_N)$, where $\alpha_N$ is a root of the polynomial
\begin{equation}\label{cubic_poly}f_N = X^3 + N X + 1\text,\end{equation}
with a positive integer $N$ such that $4 N^3 + 27$ is squarefree. By \cite{Erdos1953}, infinitely many such $N$ exist. We may assume that $\alpha_N \in \R$. From \eqref{cubic_poly}, we get
$$\frac{N^2}{N^3 + 1} < -\alpha_N = \frac{1}{\alpha_N^2 + N} < 1/N\text.$$
Since $-1/\alpha_N$ is a unit of the ring of integers of $K_N$, and $N < -1/\alpha_N < N + 1/N^2$, we have $R_K \leq \log(N + 1/N^2)$.
The discriminant $-4N^3 - 27$ of $f_N$ is squarefree by hypothesis, so $|\Delta_K| = 4 N^3 + 27$. Now we see by a simple computation that \eqref{disc_reg} holds.

In the function field case, Frei \cite{Frei2010} investigated quadratic extensions of rational global function fields.

\begin{theorem}\cite[Theorem 2]{Frei2010}\label{qualitative_function_field}
Let $K$ be a finite field, and $F$ a quadratic extension field of the rational function field $K(x)$ over $K$. Denote the integral closure of $K[x]$ in $F$ by $R$. Then the following two statements are equivalent.
\begin{enumerate}[1.]
 \item $u(R) = \omega$
 \item The function field $F|K$ has full constant field $K$ and genus $0$, and the infinite place of $K(x)$ splits into two places of $F|K$.
\end{enumerate}
\end{theorem}

This criterion can also be phrased in terms of an element generating $F$ over $K(x)$. If, for example, $K$ is the full constant field of $F$ and of odd characteristic then we can write $F = K(x,y)$, where $y^2 = f(x)$ for some separable polynomial $f \in K[x]\setminus K$. Then we get $u(R) = \omega$ if and only if $f$ is of degree $2$ and its leading coefficient is a square in $K$ (\cite[Corollary 1]{Frei2010}).

Theorem \ref{qualitative_function_field} holds in fact for arbitrary perfect base fields $K$. An alternative proof given at the end of \cite{Frei2010} implies the following stronger version:

\begin{theorem}
Let $F|K$ be an algebraic function field in one variable over a perfect field $K$. Let $S$ be a set of two places of $F|K$ of degree one, and denote by $R$ the ring of $S$-integers of $F|K$. Then $u(R) = \omega$ if and only if $F | K$ is rational.
\end{theorem}

All of the rings $R$ investigated above have in common that their unit groups are of rank at most one. Currently, there are no known nontrivial criteria for families of number fields (or function fields) whose rings of integers have unit groups of higher rank. We consider it an important direction to find such criteria.

Peth\H{o} and Ziegler investigated a modified version of Problem \ref{qual_problem}, where one asks whether a ring of integers has a power basis consisting of units \cite{Ziegler2010, Petho2011}. For example, Ziegler proved the following:

\begin{theorem}\cite[Theorem 1]{Ziegler2010}\label{ziegler_power_base_quartic}
Let $m>1$ be an integer which is not a square. Then the order $\Z[\sqrt[4]{m}]$ admits a power basis consisting of units if and only if $m=a^4 \pm 1$, for some integer $a$.
\end{theorem}

Since analogous results are already known for negative $m$ \cite{Ziegler2008} and for the rings $\Z[\sqrt[d]{m}]$, $d<4$ \cite{Belcher1974, Tichy2007}, Theorem \ref{ziegler_power_base_quartic} motivates the following conjecture:

\begin{conjecture}\cite[Conjecture 1]{Ziegler2010}
Let $d \geq 2$ be an integer and $m \in \Z \setminus \{0\}$, and assume that $\sqrt[d]{m}$ is an algebraic number of degree $d$. Then $\Z[\sqrt[d]{m}]$ admits a power basis consisting of units if and only if $m = a^d \pm 1$, for some integer $a$.
\end{conjecture}

For rings $R$ with $u(R) = \omega$, Ashrafi \cite{Ashrafi2009} investigated the stronger property that every element of $R$ can be written as a sum of $k$ units for all sufficiently large integers $k$. Ashrafi proved that this is the case if and only if $R$ does not have $\Z/2\Z$ as a factor, and applied this result to rings of integers of quadratic and complex cubic number fields.

Let $R$ be an order in a quadratic number field. Ziegler \cite{Ziegler2011} found various results about representations of elements of $R$ as sums of $S$-units in $R$, where $S$ is a finite set of places containing all Archimedean places. 

Another variant of Problem \ref{qual_problem} asks for representations of algebraic integers as sums of distinct units. Jacobson \cite{Jacobson1964} proved that in the rings of integers of the number fields $\Q(\sqrt{2})$ and $\Q(\sqrt{5})$, every element is a sum of distinct units. His conjecture that these are the only quadratic number fields with that property was proved by {\'S}liwa \cite{Sliwa1974}. Belcher \cite{Belcher1974, Belcher1975} investigated cubic and quartic number fields. A recent article by Thuswaldner and Ziegler \cite{Thuswaldner2011} puts these results into a more general framework: they apply methods from the theory of arithmetic dynamical systems to additive unit representations.

\section{The extension problem}

\begin{problem}\cite[Problem B]{Jarden2007}\label{extension_problem}
Is it true that each number field has a finite extension $L$ such that the ring of integers of $L$ is generated by its units?
\end{problem}

If $K$ is an Abelian number field, that is, $K|\Q$ is a Galois extension with Abelian Galois group, then we know by the Kronecker-Weber theorem that $K$ is contained in a cyclotomic number field $\Q(\zeta)$, where $\zeta$ is a primitive root of unity. The ring of integers of $\Q(\zeta)$ is $\Z[\zeta]$, which is obviously generated by its units. Problem \ref{extension_problem} was completely solved by Frei \cite{Frei2011B}:

\begin{theorem}\cite[Theorem 1]{Frei2011B}\label{extension_problem_theorem}
For any number field $K$, there exists a number field $L$ containing $K$, such that the ring of integers of $L$ is generated by its units.
\end{theorem}

The proof relies on finding elements of the ring of integers of $K$ with certain properties via asymptotic counting arguments, and then using these properties to generate easily manageable quadratic extensions of $K$ in which those elements are sums of units of the respective rings of integers. The field $L$ is then taken as the compositum of all these quadratic extensions.

Prior to this, with an easier but conceptually similar argument, Frei \cite{Frei2010B} answered the function field version of Problem \ref{extension_problem}:

\begin{theorem}\cite[Theorem 2]{Frei2010B}
Let $F|K$ be an algebraic function field over a perfect field $K$, and $R$ the ring of $S$-integers of $F$, for some finite set $S\neq \emptyset$ of places. Then there exists a finite extension field $F'$ of $F$ such that the integral closure of $R$ in $F'$ is generated by its units.
\end{theorem}

\section{The quantitative problem}

\begin{problem}\cite[Problem C]{Jarden2007}
Let $K$ be an algebraic number field. Obtain an asymptotic formula for the number $N_k(x)$ of positive rational integers $n \leq x$ which are sums of at most $k$ units of the ring of integers of $K$.
\end{problem}

As Jarden and Narkiewicz noticed, Lemma \ref{JN_lemma} and Szemer\'edi's theorem (see \cite{Gowers2001}) imply that
$$\lim_{x \to \infty} \frac{N_k(x)}{x} = 0\text,$$
for any fixed $k$.

A similar question has been investigated by Filipin, Fuchs, Tichy, and Ziegler \cite{Filipin2008, Filipin2008B, Fuchs2009}. We state here the most general result \cite{Fuchs2009}. Let $R$ be the ring of $S$-integers of a number field $K$, where $S$ is a finite set of places containing all Archimedean places. Two $S$-integers $\alpha$, $\beta$ are \emph{associated}, if there exists a unit $\varepsilon$ of $R$ such that $\alpha = \beta \varepsilon$. For any $\alpha \in R$, we write
$$N(\alpha) := \prod_{\nu \in S} |\alpha|_\nu\text.$$

Fuchs, Tichy and Ziegler investigated the counting function $u_{K, S}(n, x)$, which denotes the number of all classes $[\alpha]$ of associated elements $\alpha$ of $R$ with $N(\alpha) \leq x$ such that $\alpha$ can be written as a sum 
$$\alpha = \sum_{i=1}^n \varepsilon_i\text,$$
where the $\varepsilon_i$ are units of $R$ and no subsum of $\varepsilon_1 + \cdots + \varepsilon_n$ vanishes. The proof uses ideas of Everest \cite{Everest1990}, see also Everest and Shparlinski \cite{Everest1999}.

\begin{theorem}\cite[Theorem 1]{Fuchs2009}
Let $\varepsilon > 0$. Then 
$$u_{K, S}(n, x) = \frac{c_{n-1, s}}{n!}\left(\frac{\omega_K (\log x)^s}{\Reg_{K,S}} \right)^{n-1} + o((\log x)^{(n-1)s - 1 + \varepsilon})\text,$$
as $x \to \infty$. Here, $\omega_K$ is the number of roots of unity of $K$, $\Reg_{K,S}$ is the $S$-regulator of $K$, and $s = |S|-1$. The constant $c_{n, s}$ is the volume of the polyhedron
$$\{(x_{11}, \ldots, x_{ns}) \in \R^{ns} \mid g(x_{11},\ldots, x_{ns})<1\}\text,$$
with
$$g(x_{11}, \ldots, x_{ns}) = \sum_{i=1}^s \max\{0, x_{1i}, \ldots, x_{ni}\} + \max\left\{0, - \sum_{i=1}^s
x_{1i}, \ldots, -\sum_{i=1}^s x_{ni}\right\}\text.$$
\end{theorem}

The values of the constant $c_{n,s}$ are known in special cases from \cite{Fuchs2009}:
\begin{center}
\begin{tabular}{  c c c c c c }
 \hline
      & $n$ &   &   &   &    \\
\hline
  $s$ & 1   & 2 & 3 & 4  &  5  \\
\hline
  1   &  2  & 3  & 4  &  5  & 6 \\
  2  &  3  &  15/4  &  7/2  &  45/16  &  \\
  3  &  10/3  &  7/3  &  55/54  &  &  \\
  4  &  35/12  &  275/32  &  &  &  \\
  5  &  21/10  & & & & \\
\hline
\end{tabular}
\end{center}

Furthermore, $c_{n,1}=n+1$ and $c_{1,s}=\frac{1}{s!}\binom{2s}{s}$.

In the following we calculate the constant $c_{n,s}$ for $n>1$ and $s=2$. This constant is the volume of the polyhedron
$$
V = \left\lbrace  ( x,y) \in \mathbb{R}^{n} \times \mathbb{R}^{n} :  g( x,y)<1  \right\rbrace ,
$$
with
$$
g( x,y)=\max_{i} \left\lbrace 0,x_i \right\rbrace +\max_{i} \left\lbrace 0,y_i \right\rbrace+\max_{i} \left\lbrace 0,-x_i-y_i \right\rbrace ,
$$
where $ x=(x_1, \ldots , x_n)$, $ y=(y_1, \ldots , y_n)$.

For any $K$, $L$, $M \in \left\lbrace 1, \ldots , n \right\rbrace $ we consider the sets
$$
V_{K,L,M} = \left\lbrace  ( x,y) \in \mathbb{R}^{2n} : x_i\leq x_K, \; y_i\leq y_L, \; x_M+y_M\leq x_i +y_i, \; g( x,y)<1  \right\rbrace .
$$
Clearly the union of these sets is $V$ and the intersection of any two of them has volume zero. Thus
$$
c_{n,2} = \sum_{K=1}^{n} \sum_{L=1}^{n} \sum_{M=1}^{n} I_{K,L,M} ,
$$
where $I_{K,L,M}$ is the volume of $V_{K,L,M}$. For the values of $I_{K,L,M}$ we distinguish three cases:
\renewcommand{\labelenumi}{(\roman{enumi})}
\begin{enumerate}
 \item $K,L,M$ are pairwise distinct;
 \item exactly two of the indices $K,L,M$ are equal;
 \item $K=L=M$.
\end{enumerate}
The third case is simple. Since $x_i \leq x_K$, $y_i \leq y_K$ implies $x_i+y_i \leq x_K+y_K$ we obtain $x_i+y_i = x_K+y_K$. Thus $V_{K,K,K}$ has volume zero.

We only have to consider the remaining cases (i) and (ii).
Clearly,
$$
c_{n,2} = n(n-1)(n-2)I_{1,2,3} + 3 n(n-1) I_{1,1,2} .
$$

\def\thesubsection{\arabic{section}.\roman{subsection}}

\subsection{Calculation of $I_{1,2,3}$}

This case can only happen if $n \geq 3$. The inequalities $x_3+y_3\leq x_i+y_i$ give us lower bounds for $x_i$ and $y_i$ and we always have the upper bounds $x_i \leq x_1$ and $y_i \leq y_2$. Hence we have
$$
x_3+y_3-x_i \leq y_i \leq y_2
$$
and
$$
x_i \leq x_1.
$$
Note that
$$
g( x,y)=\max \left\lbrace 0,x_1 \right\rbrace +\max \left\lbrace 0,y_2 \right\rbrace+\max \left\lbrace 0,-x_3-y_3 \right\rbrace .
$$
We integrate with respect to the $y_i$'s, $i\neq 2,3$ and obtain
$$
I_{1,2,3}= \idotsint_{\substack{x_3+y_3-x_i \leq y_i \leq y_2\\
 x_i \leq x_1 , \; g( x,y)<1
}}   dx dy      = \idotsint_{\substack{ x_3+y_3\leq x_2+y_2 \\x_3+y_3-y_2 \leq x_i \leq x_1 \\  y_3 \leq y_2 ,\; g( x,y)<1
}}  \prod_{j\neq 2,3} (y_2-x_3-y_3+x_j)   dx dy_2 dy_3 .
$$
Next we integrate over the $x_i$'s, $i \neq 1,2,3$ and obtain

$$
I_{1, 2, 3} = \idotsint_{\substack{x_2,x_3\leq x_1, \; y_3 \leq y_2\\
x_3+y_3 \leq x_2+y_2\\
g( x,y)<1
}}   \dfrac{1}{2^{n-3}}   \left( y_2-x_3-y_3+x_1  \right)^{2n-5}dx_1dx_2dx_3dy_2dy_3.
$$ 

For the values of $g(x,y)$ we consider the following cases depending on the signs of $x_1$, $y_2$ and $-x_3-y_3$:

\begin{center}
\begin{tabular}{ | c | c | c | c | c | }
 \hline
$r$ & $x_1$ & $y_2$ & $-x_3-y_3$ & $g(x,y)$ \\
\hline
1 & $\geq 0$ & $<0$ & $<0$ & $x_1$ \\
2 & $<0$ & $\geq 0$ & $<0$ & $y_2$ \\
3 & $<0$ & $<0$ & $\geq 0$ & $-x_3-y_3$ \\
4 & $\geq0$ & $\geq0$ & $<0$ & $x_1+y_2$ \\
5 & $\geq0$ & $<0$ & $\geq0$ & $x_1-x_3-y_3$ \\
6 & $<0$ & $\geq0$ & $\geq0$ & $y_2-x_3-y_3$ \\
7 & $\geq0$ & $\geq0$ & $\geq0$ & $x_1+y_2-x_3-y_3$ \\
\hline
\end{tabular}
\end{center}

According to the table we split the integral into seven parts:
$$
I_{1,2,3}=\sum_{r=1}^{7} I_{1,2,3}^{(r)}.
$$

One can calculate these integrals with the help of a computer algebra system. We just give the final expressions:
\begin{eqnarray*}
I_{1,2,3}^{(1)}  =I_{1,2,3}^{(2)}  =I_{1,2,3}^{(3)} & =  & \dfrac{2}{n(2n-1)(n-1)2^n},\\
I_{1,2,3}^{(4)}  =I_{1,2,3}^{(5)}  =I_{1,2,3}^{(6)}& = & \dfrac{2}{n(n-1)2^n}, \\
I_{1,2,3}^{(7)} & = & \dfrac{2}{n2^n}.
\end{eqnarray*}

In conclusion we have
$$
I_{1,2,3}= \dfrac{2(n+1)(2n+1)}{n(2n-1)(n-1)2^n}.
$$

\subsection{Calculation of $I_{1,1,2}$}

We proceed in the same way as in the other case. We have the same bounds
$$
x_2+y_2-x_i \leq y_i \leq y_1
$$
and
$$
x_i \leq x_1.
$$
We integrate first with respect to the $y_i$'s and then with respect to the $x_i$'s, $i \neq 1,2$, and obtain
$$
I_{1,1,2}     = \idotsint_{\substack{x_2+y_2-y_1 \leq x_i \leq x_1 \\  y_2 \leq y_1 , \; g( x,y)<1
}}  \prod_{j\neq 1,2} (y_1-x_2-y_2+x_j)   dx dy_1 dy_2 =
$$
$$
=\idotsint_{\substack{x_2\leq x_1 , \; y_2 \leq y_1\\
g( x,y)<1}}   \dfrac{1}{2^{n-2}}   \left( y_1-x_2-y_2+x_1  \right)^{2n-4} dx_1dx_2dy_1dy_2 .
$$
Proceeding as in the previous section we again split the integral into seven parts $I_{1,1,2}^{(r)}$, $r=1, \ldots ,7$, and obtain:

\begin{eqnarray*}
I_{1,1,2}^{(1)}  =I_{1,1,2}^{(2)}  =I_{1,1,2}^{(3)} & =  & \dfrac{1}{n(2n-1)(n-1)2^n},\\
I_{1,1,2}^{(4)}  =I_{1,1,2}^{(5)}  =I_{1,1,2}^{(6)}& = & \dfrac{1}{n(n-1)2^n}, \\
I_{1,1,2}^{(7)} & = & \dfrac{1}{n2^n}.
\end{eqnarray*}

Hence
$$
I_{1,1,2}= \dfrac{(n+1)(2n+1)}{n(2n-1)(n-1)2^n}.
$$
\begin{conclusion} 
The value of $c_{n,2}$ is
$$
\dfrac{(n+1)(2n+1)}{2^n}.
$$
\end{conclusion}

\begin{remark}
The computation of $c_{n,s}$ for $s>2$ seems to be more difficult and might be considered later.
\end{remark}

\def\thesubsection{\arabic{section}.\arabic{subsection}}

\section{Matrix rings}

\subsection{Matrix rings over arbitrary rings}

Let $R$ be any ring with 1. We say that two elements $a,b \in R$ are equivalent ($a\sim b$) if there exist two units $u,v \in R^\times$ such that $b=uav$. V\'amos \cite[Lemma 1]{Vamos2005} already noticed the following simple fact.
\begin{lemma} \label{lemeq}
Let $R$ be a ring and $a,b \in R$. If $a \sim b $ then, for all $k \geq 1$, $a$ is $k$-good if and only if $b$ is $k$-good.
\end{lemma}
 
We consider the ring $M_n(R)$ of $n \times n$ matrices, with $n \geq 2$, over an arbitrary ring $R$ with 1. As usual $GL_n(R)$ denotes the group of units of $M_n(R)$.

For $a \in R$ the matrix $E_n(a,i,j)$, $i$, $j \in \left\lbrace 1 ,\ldots,n \right\rbrace $, $i \neq j$, is the $n \times n$ matrix with 1 entries on the main diagonal, $a$ as the entry at position $(i,j)$ and 0 elsewhere. We call this kind of matrices \textit{elementary matrices} and denote by $E_n(R)$ the subgroup of $GL_n(R)$ generated by elementary matrices, permutation matrices and $-I$, where $I$ is the identity matrix of $M_n(R)$.

Let us consider a more specific kind of $k$-goodness introduced by V\'amos \cite{Vamos2005}.

\begin{definition}
A square matrix of size $n$ over $R$ is \textit{strongly $k$-good} if it can be written as a sum of $k$ elements of $E_n(R)$. The ring $M_n(R)$ is strongly $k$-good if every element is strongly $k$-good.
\end{definition}

The following lemma is Lemma 1 from \cite{Henriksen1974} and Lemma 5 from \cite{Vamos2005}.

\begin{lemma} \label{lemHen}
Let $R$ be a ring and $n \geq 2$. Then any diagonal matrix in $M_n(R)$ is strongly 2-good.
\end{lemma}

A ring $R$ is called an \textit{elementary divisor ring} (see \cite{Kaplansky1949}) if every matrix in $M_n(R)$, $n \geq 2$, can be diagonalized. Lemma \ref{lemHen} implies that, in this case, $M_n(R)$ is 2-good. In particular, if any matrix in $M_n(R)$ can be diagonalized using only matrices in $E_n(R)$ then $M_n(R)$ is strongly 2-good.

The following two remarks can be deduced without much effort from the proof of Lemma \ref{lemHen} that is given in \cite{Vamos2005}.

\begin{remark}
If $R$ is an elementary divisor ring and $1 \neq -1$ then the representation of a matrix in $M_n(R)$ as a sum of two units is never unique.

\end{remark}

\begin{remark}
If $R$ is an elementary divisor ring and $1 \neq -1$ then every element of $M_n(R)$ has a representation as a sum of two distinct units.
\end{remark}

As we have already mentioned, Henriksen \cite{Henriksen1974} proved that $M_n(R)$, where $R$ is any ring, is $3$-good.
Henriksen's result was generalized by V\'amos \cite{Vamos2005} to arbitrary dimension:

\begin{theorem}\cite[Theorem 11]{Vamos2005}
Let $R$ be a ring and let $F$ be a free $R$-module of rank $\alpha$, where $\alpha \geq 2$ is a cardinal number. Then the ring of endomorphisms $E$ of $F$ is $3$-good.

If $\alpha$ is finite and $R$ is $2$-good or an elementary divisor ring then $E$ is $2$-good. If $R$ is any one of the rings $\Z[X]$, $K[X, Y]$, $K\langle X, Y\rangle$, where $K$ is a field, then $u(E) = 3$\text. Here $K\langle X, Y\rangle$ is the free associative algebra generated by $X$, $Y$ over $K$.
\end{theorem}

To prove that a matrix ring over a certain ring has unit sum number 3, V\'amos used the following proposition.

\begin{proposition}\cite[Proposition 10]{Vamos2005} \label{propVam}
Let $R$ be a ring, $n\geq 2$ an integer and let $L=Ra_1+\cdots +Ra_n$ be the left ideal generated by the elements $a_1, \ldots , a_n \in R$. Let $A$ be the $n \times n$ matrix whose entries are all zero except for the first column which is $\left( a_1, \ldots , a_n \right)^T $. Suppose that
\begin{enumerate}[1.]
 \item $L$ cannot be generated by fewer than $n$ elements, and
\item zero is the only 2-good element in $L$.
\end{enumerate}
Then $A$ is not 2-good.
\end{proposition}

We now apply Lemma \ref{lemHen} to a special case.
Let $R$ be a ring and suppose there exists a function
$$
f: R\setminus \{ 0 \} \rightarrow \mathbb{Z}_{\geq0},
$$
with the following property: for every $a,b \in R$, $b \neq 0$, there exist $q_1,q_2,r_1,r_2 \in R$ such that
\begin{eqnarray*}
 a=q_1b+r_1, & \mbox{ where } & r_1=0 \mbox{ or } f(r_1) < f(b) , \\
a=bq_2+r_2, & \mbox{ where } & r_2=0 \mbox{ or } f(r_2) < f(b) .
\end{eqnarray*}
Then we say that $R$ has \emph{left and right Euclidean division}.

The next theorem is a generalization of the well known fact that every square matrix over a Euclidean domain is diagonalizable. The proof strictly follows the line of the one in the commutative case (see Section 3.5 of \cite{Goodman1998}), hence it is omitted.

\begin{theorem}
Let $R$ be a ring with left and right Euclidean division and $n\geq 2$. For every $A\in M_n(R)$ there exist two matrices $U,V \in E_n(R)$ such that
$$
UAV=D,
$$
where $D$ is a diagonal matrix.
\end{theorem}

\begin{corollary}
Let $R$ be a ring with left and right Euclidean division and $n\geq 2$. Then $M_n(R)$ is strongly 2-good.
\end{corollary}

We apply the previous result to the special case of quaternions.
Consider the quaternion algebra
$$
Q=\left\lbrace  a+bi+cj+dk \mid a \text{, } b\text{, } c \text{, } d \in \mathbb{Q} \text{, } i^2=-1\text{, } j^2=-1 \text{, }  k=ij=-ji \right\rbrace \text.
$$

\begin{definition}
The ring of \textit{Hurwitz quaternions} is defined as the set
$$
H= \left\lbrace  a+bi+cj+dk \in Q \ \  \mbox{    s. t.    } \ \ a,b,c,d \in \mathbb{Z} \ \ \mbox{  or  }  \ \ a,b,c,d \in \mathbb{Z}+\frac{1}{2}  \right\rbrace .
$$
\end{definition}

For basic properties about Hurwitz quaternions see \cite[Chapter 5]{Conway2003}.

In $Q$ the ring of Hurwitz quaternions plays a similar role as maximal orders in number fields.

The units of $H$ are the 24 elements $\pm 1$, $\pm i$, $\pm j$, $\pm k$ and $(\pm 1 \pm i \pm j \pm k)/2 $, so $u(H)=\omega$.

It is well known that $H$ has left and right Euclidean division. Therefore, we get the following corollary.

\begin{corollary}
 For $n \geq 2$, $M_n(H)$ is strongly 2-good.
\end{corollary}

\subsection{Matrix rings over Dedekind domains}

Let $R$ be a ring and $A$ an $r \times c$ matrix. The \textit{type} of $A$ is the pair $(r,c)$ and the \textit{size} of $A$ is $\max(r,c)$.
Let $A_1$ and $A_2$ be matrices of type $(r_1,c_1)$ and $(r_2,c_2)$, respectively. The \textit{block diagonal sum} of $A_1$ and $A_2$ is the block diagonal matrix
$$
diag(A_1,A_2)= \left[ 
\begin{array}{cc}
 A_1 & 0     \\
 0   & A_2
\end{array}
\right] ,
$$
of type $(r_1+r_2,c_1+c_2)$. A matrix of positive size is \textit{indecomposable} if it is not equivalent to the block diagonal sum of two matrices of positive size.

In 1972 Levy \cite{Levy1972} proved that, for a Dedekind domain $R$, the class number, when it is finite, is an upper bound to the number of rows and columns in every indecomposable matrix over $R$. V\'amos and Wiegand \cite{Vamos2010} generalized Levy's result to Pr\"{u}fer domains (under some technical conditions) and applied it to the unit sum problem.

\begin{theorem}(see \cite[Theorem 4.7]{Vamos2010})
 Let $R$ be a Dedekind domain with finite class number $c$. For every $n \geq 2c$, $M_n(R)$ is 2-good.
\end{theorem}

Unfortunately we do not know a criterion. The only sufficient condition we know for a matrix not to be 2-good is given by Proposition \ref{propVam}. For rings $R$ of algebraic integers this proposition is of limited use. Since ideals in Dedekind domains need at most 2 generators, condition (1) can be fulfilled only for $n=2$. Concerning condition (2) it is not hard to see that, if the unit group is infinite, there is a nonzero sum of two units in every nonzero ideal in a ring of algebraic integers. Therefore we can apply Proposition \ref{propVam} only to the non-PID complex quadratic case.

\begin{corollary}\cite[Example 4.11]{Vamos2010}
Let $R$ be the ring of integers of $\Q(\sqrt{-d})$, where $d>0$ is squarefree and $R$ has class number $c>1$. Then $u(M_2(R))=3$ and $u(M_n(R))=2$ for every integer $n \geq 2c$.
\end{corollary}

\begin{question} \cite[Example 4.11]{Vamos2010}
With the hypotheses of the previous corollary, what is the value of $u(M_n(R))$ for $3\leq n < 2c$?
\end{question}

\begin{question} \cite[Question 4.12]{Vamos2010}
If $R$ is any ring of algebraic integers with class number $c$, what is the value of $u(M_n(R))$ for $2 \leq n < 2c$?
\end{question}

\bibliographystyle{plain}
\bibliography{bibliography}

\end{document}